\documentclass[11 pt, oneside, a4paper]{amsart}

\usepackage{amsfonts, amsmath, amssymb, amsthm, amsopn, latexsym, graphicx, mathrsfs,  verbatim, enumerate, url, stmaryrd, tikz} 
\usepackage[pagebackref]{hyperref}
\usepackage[all]{hypcap} 
 \usepackage{subcaption}
\usepackage[paper=a4paper]{geometry}

\newcommand{\R}{\mathbb{R}}
\newcommand{\C}{\mathbb{C}}
\newcommand{\D}{\mathbb{D}}
\newcommand{\N}{\mathbb{N}}

\newcommand{\Q}{\mathbb{Q}}
\newcommand{\A}{\mathbb{A}}

\renewcommand{\P}{\mathbb{P}}

\newcommand{\nA}{\mathbb{A}}

\newcommand{\nP}{\mathbb{P}}

\newcommand{\cG}{\mathcal{G}}

\newcommand{\norm}[1]{\left\|#1\right\|}

\newcommand{\mc}{\mathcal}
\newcommand{\ol}{\overline}
\newcommand{\wt}{\widetilde}

\newcommand{\eps}{\varepsilon}

\newcommand{\loc}{\mathrm{loc}}

\newcommand{\abs}[1]{{\left|{#1}\right|}}
\newcommand{\on}[1]{\operatorname{#1}}

\renewcommand{\Re}{\on{Re}}

\newcommand{\set}[1]{{\left\{#1\right\}}}

\newcommand{\nk}{\Bbbk}

\DeclareMathOperator{\Gal}{Gal}	
\DeclareMathOperator{\an}{an}	
\DeclareMathOperator{\alg}{alg}	
	
\DeclareMathOperator{\spec}{Spec}

\newcommand{\fr}{\partial}

\newcommand{\rest}[1]{ \arrowvert_{#1}}

\newcommand{\unsur}[1]{\frac{1}{#1}}

\newcommand{\lrpar}[1]{\left(#1\right)}

\newcommand{\la}{\lambda}

\DeclareMathOperator{\essmin}{\mathrm{essmin}}

\newcommand{\inv}{^{-1}}

\newtheorem{introthm}{Theorem}
\newtheorem{introcor}[introthm]{Corollary}

\newtheorem{thm}{Theorem}[section]
\newtheorem{lem}[thm]{Lemma}
\newtheorem{cor}[thm]{Corollary}
\newtheorem{prop}[thm]{Proposition}

\theoremstyle{definition}
\newtheorem{rmk}[thm]{Remark}
\newtheorem{rmk-intro}{Remark}

\theoremstyle{definition}

\title[DMM conjecture for plane  polynomial maps]{On the dynamical Manin-Mumford conjecture for 
plane  polynomial maps}

\author{Romain Dujardin}
\address{Sorbonne Universit\'e, Laboratoire de probabilit\'es, statistique et mod\'elisation, UMR 8001,  4 place Jussieu, 75005 Paris, France}
\email{\href{romain.dujardin@sorbonne-universite.fr}{romain.dujardin@sorbonne-universite.fr}}

\author{Charles Favre}
\address{CNRS - Centre de Math\'ematiques Laurent Schwartz, 
	\'Ecole Polytechnique, 
	91128 Palaiseau Cedex, France}
\email{\href{charles.favre@polytechnique.edu}{charles.favre@polytechnique.edu}}

\author{Matteo Ruggiero}
\address{Université de Paris and Sorbonne Université, CNRS, Institut de Mathématiques de Jussieu-Paris Rive Gauche, F-75006 Paris, France}
\email{\href{mailto:matteo.ruggiero@imj-prg.fr}{matteo.ruggiero@imj-prg.fr}}

\date{\today}
\thanks{The research activities of the authors are  partially supported by the ANR grants Fatou (ANR-17-CE40-0002-01) and   DynAtrois (ANR-24-CE40-1163)}

\begin{document}

\begin{abstract}
We prove the dynamical Manin-Mumford conjecture 
for regular polynomial maps of $\A^2$ and irreducible curves
avoiding super-attracting orbits at infinity, over any field of characteristic $0$. 
\end{abstract} 

 \maketitle

\setcounter{tocdepth}{1}
\tableofcontents

\section*{Introduction}

The dynamical Manin-Mumford conjecture for polarized endomorphisms of algebraic varieties, first 
formulated by S.-W. Zhang in two influential papers~\cite[Conjecture 2.5]{zhang:smallpointsadelicmetrics}
and~\cite[Conjecture 1.2.1]{zhang:distributions}, has been a driving force for the development of 
the field of arithmetic dynamics. It was realized by Ghioca, Tucker and Zhang~\cite{GTZ11} and 
Pazuki~\cite{pazuki} that the original formulation of the conjecture was too optimistic, 
and a modified conjecture was proposed in~\cite{GTZ11} and more recently in~\cite{GT21}. It can be 
 stated as follows: \emph{let 
$f:X\to X$ be a polarized endomorphism of a smooth projective variety over a field of characteristic zero, and $Z\subset X$ be a subvariety containing a Zariski dense set of preperiodic points. 
Then either $Z$ is preperiodic or $Z$ is special}, in the sense that it is contained in some 
subvariety $Y$ that is both $f^n$- and $\psi$-invariant, for some $n\geq 1$, 
where $\psi$ is another polarized endomorphism 
commuting with $f^n$, and $Z$ is preperiodic under $\psi$. Recall that an endomorphism is said to be polarized if there is an ample 
line bundle $L\to X$ such that $f^*L \simeq L^d$ for some $d\geq 2$. A basic  example is that 
of non-invertible endomorphisms of $\P^k$, for which we can take $L  = \mathcal O(1)$ and $d$ is the degree of $f$. 

Despite its importance, very few cases of the conjecture have been settled so far. 
One first case is of course the original Manin-Mumford conjecture, 
which was solved by Raynaud~\cite{raynaud:ManinMumfordcourbes, raynaud:ManinMumfordsousvar}.
Viewed as a dynamical statement, it 
 deals with endomorphisms of Abelian varieties,  
 and was generalized to related  settings, such as commutative algebraic groups (see Hindry~\cite{hindry:1988}, and also Lang's classical paper~\cite{lang:division}). 
 Uniform versions involving height bounds were subsequently obtained by  
 S.-W. Zhang, David and Philippon, Chambert-Loir and others. We refer to the recent work of Kühne~\cite{MR4404794} in the semi-abelian case for a latest update, 
and more references.  
Closer to algebraic dynamics is the case 
 of polarized endomorphisms of $(\P^1)^k$, which was solved by Ghioca, Nguyen and Ye~\cite{GNY1, GNY2} (see also~\cite{GT21}). Note that such mappings are of product type, 
that is  of the form $f(x_1, \ldots , x_k) = (f_1(x_1), \ldots , f_k(x_k))$, 
so their dynamical complexity reduces to dimension 1, which is a key step in the proof.  

{Let us also note that 
the counterexamples of~\cite{GTZ11} to the original formulation of the conjecture (i.e., situations where the subvariety $Z$ is not preperiodic), occurring on products of elliptic curves, yield counterexamples of product type  in $(\P^1)^2$ via the Lattès construction.
By taking the Segre embedding to $\P^3$ and homogeneous extension of these examples (see \cite{fakhruddin:questselfmapalgvar,bhatnagar-szpiro:vamppolarextprojspace}), we get counterexamples that are endomorphisms of $\nP^3$.
However these product counterexamples are never regular on the birationally equivalent model $\nP^2$, where no regular counterexamples are known to us.
}

In this paper we will establish the dynamical 
Manin-Mumford conjecture for a wide class of 2-dimensional examples, 
whose dynamical behavior is  truly higher dimensional. 

Note that besides polarized endomorphisms, a partial answer to the conjecture 
was given in~\cite{DMM-henon} for plane polynomial automorphisms. 

\medskip

Let $\nk$ be any field of characteristic zero. 
If required we fix an algebraic closure $\nk^{\rm alg}$ of $\nk$. 
Let $f\colon\nA_\nk^2 \to \nA_\nk^2$ be a polynomial self-map 
of the affine plane of degree $d \geq 2$, 
which in coordinates is written as
\[
f(z,w)=\big(P(z,w),Q(z,w)\big)\text,
\]
with $P,Q \in \nk[x,y]$.
We say that $f$ is \emph{regular} if 
 it extends to an endomorphism of $\nP_\nk^2$ of degree $d\ge2$; this means that 
 $P=P_d + \text{l.o.t.}$ and $Q=Q_d + \text{l.o.t.}$, where $P_d$ and $Q_d$ are homogeneous polynomials of degree $d$ without common factors. 
 In particular $f$ induces a rational map $f_\infty := [P_d: Q_d]$ on the line at infinity, which is fixed. Note that 
 any endomorphism of $\nP_\nk^2$ with a totally invariant line is conjugate to a 
 regular polynomial map, and that a generic polynomial map of $\A^2$ whose components are polynomials of degree $\le d$ is regular. 
 
For regular polynomial maps of $\nA^2$, it seems that none of the known obstructions to the dynamical Manin-Mumford conjecture can arise. As said above, according to~\cite{GT21}, 
one obstruction would be the existence of a preperiodic  curve $C$ for some endomorphism $\psi$ commuting with $f$. Such pairs $(f,\psi)$ were classified in~\cite{dinh:commuting} over $\C$ (see also \cite{dinh-sibony:commuting, kaufmann}), and after a ramified cover they are all induced from a product map or a monomial map on the multiplicative $2$-torus. 
Thanks to~\cite{GNY1}, in any such case $C$ must be also $f$-preperiodic.
 Thus we expect that the dynamical Manin-Mumford conjecture 
holds unconditionally in this setting. 
Our main result confirms this expectation in the vast majority of cases.

\begin{introthm}\label{thm:DMM}
Let $f$ be a regular polynomial endomorphism of $\nA^2$ of degree $d \geq 2$, 
defined over an arbitrary  field $\nk$ of characteristic $0$. 

Let  $C \subset \nA^2$ be an irreducible algebraic curve containing infinitely many preperiodic points of $f$, and suppose that 
the  closure of $C$   in $\nP^2$ contains a point $p\in L_\infty$ 
which is not eventually superattracting. 
Then $C$ is  preperiodic under $f$. 
\end{introthm}

\begin{rmk-intro}
{When $\nk$ is a number field, our proof requires only the existence of a sequence of points $p_n$ in $C$ whose canonical dynamical height
tends to $0$ (see Theorem~\ref{thm:main_precised} below),
thereby proving a version of the dynamical Bogomolov conjecture: 
for some $\eps>0$, the set of points in $C$ of dynamical height smaller than $\eps$ is finite.}
\end{rmk-intro}

\begin{introcor}\label{cor:DMM}
Under the assumptions of the theorem, if $f$ has no super-attracting points on the line at infinity, then 
any irreducible algebraic curve $C \subset \P^2$
containing  infinitely many periodic points is preperiodic, that is, the dynamical Manin-Mumford conjecture holds for $f$.
\end{introcor}

 Our strategy relies on techniques from Arakelov geometry, in particular on the notion 
of canonical height, which is now classical in arithmetic dynamics, together with a variety of techniques from 
holomorphic and non-Archimedean dynamics. 

 We first assume that 
 $\nk$ is a number field, in which case we prove a stronger Bogomolov-type statement (Theorem~\ref{thm:main_precised}). 
We first recall in Section~\ref{sec:heights}, that  there exists  a canonical height
$h_f$  on $\A^2(\nk^{\alg})$, 
for which preperiodic points are exactly points of zero height. 
By a theorem of Zhang, all points of $C$ lying at infinity are  preperiodic, hence, 
by replacing $C$ by some iterate 
we may assume that they are periodic, and,  
thanks to our assumptions, one of these periodic points  
 is not superattracting for $f_\infty$.  Fix such a   point $p\in \overline C \cap L_\infty$.
A second consequence of Zhang's theorem is that or each place $v$ of $\nk$,
 the dynamical Green function $g_{f,v}(z,w) := \lim_{n\to\infty} \frac1{d^n} \log^+ \norm{f^n(z,w)}$
satisfies 
\begin{equation} \label{eq:energy}
\int_{C} g_{f,v} dd^c g_{f,v}=0.
\end{equation} 
To make sense of this equality at
  a finite place,  we consider  the Berkovich analytification of $C$, in which case $dd^c  g_{f,v}$ is 
 the Laplacian of the subharmonic function 
$g_{f,v}|_C$ in the sense of Thuillier, see~\cite{thuillier:phdthesis}. 

For an appropriate choice of $v$, 
 we may suppose that  the  periodic
  point  $p$   is either repelling or parabolic for $f_\infty$. 
 In both cases, we construct in Section~\ref{sec:super-stable}  
 a local super-stable manifold $W^{\rm ss}_{\loc}(p)$. 
  When $p$ is repelling and $v$ is Archimedean this is classical. 
We extend this construction to  the non-Archimedean setting, and allow for non-repelling $p$ 
(see Theorem~\ref{thm:super-stable_manifold}). 
Then, in both (repelling or parabolic) cases we establish 
 graph transform type estimates which will be useful for the local analysis of the Green function at $p$. 
 When $p$ is parabolic this borrows from the work of  Hakim~\cite{hakim}. 

In Section~\ref{sec:conclusion_number_field}  we combine these estimates with condition~\eqref{eq:energy} to show that, near $p$, 
   $C$ must locally coincide with $W^{\rm ss}_{\loc}(p)$, 
 which finally implies that  $C$ is periodic.   
Note that a similar argument appears in the recent work~\cite{gauthier-taflin-vigny:sparsity}.

A key point in this argument is that at the chosen place $v$, $p$ belongs to the Julia set of $f_\infty$. 
This is no longer true when $f_\infty$ is superattracting at $p$, and we are not able to conclude in this case. 
{Still, this situation leads to interesting dynamical considerations and objects, which we explore in~\cite{DFR2}.}

Finally, we develop a specialization argument in Section~\ref{sec:spec} to reduce Theorem~\ref{thm:DMM} for an arbitrary $\nk$
 to the case where $\nk$ is a number field.  So here we rather deal with an algebraic family of 
endomorphisms of $\P^2$ parameterized by some
algebraic variety $S$. 
The main issue is to ensure that for such a family, an infinite set of preperiodic points cannot shrink to a finite set 
 too often on  $S$. To do so, one needs
to control collisions of periodic points (using the 
Shub-Sullivan  theorem~\cite{shub-sullivan} in the spirit of~\cite{DMM-henon}); 
and the splitting of local preimages of a super-attracting cycle, 
a phenomenon that was studied in particular by Chio and Roeder~\cite{chio-roeder}.

{Note that 
an alternative approach to  the proof of Theorem~\ref{thm:DMM} for   arbitrary $\nk$ is to use  Moriwaki's  height theory as presented in~\cite{zbMATH07150972}.} 

\begin{rmk-intro}
{Several interesting developments appeared after the first release of this paper. 
\begin{enumerate}
\item A far-reaching generalization of the dynamical Manin-Mumford conjecture was formulated by DeMarco and Mavraki~\cite{DeMarco-Mavraki:DMM}.
\item  When the restriction of the map $f$ has no exceptional point, then J. Xie~~\cite[Theorem~1.20]{MR4711366} 
proved that there can be only finitely many $f$-invariant curves, except if $f$ is homogenous. More recently, Zhong~\cite[\S~7]{arXiv:2508.13873} gave 
a  classification of regular polynomial endomorphisms of $\nA^2$ having infinitely many periodic curves of the same degree.
It is based on former results by Dabija-Jonsson~\cite{MR2384901} and 
Favre-Pereira~\cite{MR3421725}. 
\item In view of the recent results of Gauthier-Taflin-Vigny~\cite[Theorem~D]{gauthier-taflin-vigny:sparsity}, it is natural to expect the following statement to be true: 
there exists an open Zariski dense subset $U$ of the space of regular polynomial endomorphism of $\nA^2$ of a fixed degree, and for any integer $\delta$ a constant $C(\delta)>0$ such that 
for any $f\in U$ and any curve of degree at most $\delta$, then $C$ contains at most $C(\delta)$ preperiodic points.
 The case $\deg(f)=2$ and $C$ is the critical locus of $f$ is treated in op.cit., relying on 
  sophisticated dynamical and arithmetic  techniques.
\end{enumerate}
}
\end{rmk-intro}

\subsection*{Ackowledgements} 
{We thank the referees for their careful reading and 
 constructive remarks.  
 }
 
\section{Dynamical heights} \label{sec:heights}

In this section, we recall some basic facts on canonical heights attached to 
endomorphisms of the projective plane defined over a number field. Our purpose is to establish 
Proposition~\ref{prop:Green_harmonic}, which is the key arithmetic geometry input in our main theorem.

Throughout this section, we assume that  $\nk$ is 
a number field. 

\subsection{Vocabulary of number fields}
We denote by $M_\nk =\{ v\}$ the set of places of $\nk$, that is,  the set of all 
multiplicative norms $|\cdot|_v$ on $\nk$ that restrict to either the standard euclidean norm $|\cdot|_\infty$, or to a $p$-adic norm
$|\cdot|_p$ on $\Q$ for some prime number $p>1$. We normalize the $p$-adic norm by 
$|p|_p= p^{-1}$. We let $\nk_v$ be the completion of $\nk$ w.r.t. $|\cdot|_v$, and write $n_v:= [\nk_v:\Q_v]$. 
The product formula asserts that  
for any $a \in \nk$, we have
\[
  \prod_{v\in M_\nk}\abs{a}^{n_v}_v=1~.
\]
The set $M_\nk$ splits into the finite set $M^\infty_\nk$ of Archimedean places (whose restriction to $\Q$ is $|\cdot|_\infty)$, 
and the set of finite (or non-Archimedean) places. 

When $v\in M^\infty_\nk$, the algebraic closure 
$\C_v$ of $\nk_v$ is isometric to $\C$. When $v$ is a finite place extending $|\cdot|_p$ on $\Q$, then $|\cdot|_v$ extends canonically to 
$\nk_v^{\alg}$, and its completion $\C_v$ 
is both complete and algebraically closed.

\subsection{Regular polynomial maps}
Let $(z,w)$ be affine coordinates on the affine plane $\nA_\nk^2$. 
We also consider homogeneous coordinates $[z_0:z_1:z_2]$ on the projective plane $\nP_\nk^2$
and identify the affine plane $\nA_\nk^2$ with the Zariski open set $z_0\neq 0$ so that 
$z = z_1/z_0$ and $w= z_2/z_0$. We denote by $L_\infty = \{ z_0 =0\}$ the line at infinity.

Let $f \colon \nA_\nk^2 \to \nA_\nk^2$ be any polynomial self-map of the affine plane of degree $d \geq 2$
that extends to an endomorphism of $\nP_\nk^2$. In the coordinates $z,w$, it is 
given by
\[
f(z,w)=\big(P(z,w),Q(z,w)\big)\text,
\]
where $P,Q \in \nk[z,w]$ satisfy $\max\{\deg P, \deg Q\}=d$.
The fact that $f$ extends to a regular endomorphism $\overline{f} \colon \nP_\nk^2\to \nP_\nk^2$ is equivalent to say that 
$P=P_d + \text{l.o.t.}$ and $Q=Q_d + \text{l.o.t.}$, 
where $P_d$ and $Q_d$ are homogeneous polynomials of degree $d$ without common factors.

For $n\in \N$ we write 
$f^n(z,w)=\big(P_n(z,w),Q_n(z,w)\big)$.
The restriction of $f_\infty$ to $L_\infty$ is an endomorphism of $\nP^1_\nk$
given in homogeneous coordinates by
\[f_\infty([z_1:z_2])= [P_d(z_1,z_2):Q_d(z_1,z_2)]~.\]

\subsection{Green functions}\label{sec:local green} 
The next proposition follows from the Nullstellensatz (see e.g.~\cite[Theorem 3.11]{silverman:arithdynsys}).
\begin{prop}
For any $v\in M_\nk$, there exists a constant $C_v \ge 1$ such that 
\begin{equation}\label{eq:stand}
C_v^{-1}
\le
\frac{ \max\{1,|P(z,w)|_v,|Q(z,w)|_v\}}
{\max\{1,|z|_v,|w|_v\}^d}
\le
C_v
\end{equation}
for all $z,w \in \C_v$. 
Moreover, for all but finitely many $v\in M_\nk$ we may take $C_v = 1$.
\end{prop}
By the previous proposition, the sequence of functions
\[g_{v,n}(z,w):= \frac1{d^n}\log  \max\{1,|P_n(z,w)|_v, |Q_n(z,w)|_v\}\]
converges uniformly on $\C^2_v$ to a continuous function $g_v$, and
the next proposition follows. 
\begin{prop}\label{prop:green}
For any $v\in M_\nk$, the function $g_v \colon \C_v^2 \to \R_+$
is continuous, it satisfies the invariance equation
$g_v \circ f = d g_v$, and we have
\[
\big| g_v(z,w) - \log \max\{1,|z|_v,|w|_v\} \big| \le \frac{\log C_v}{d-1}
~.\]
The set $\{(z,w) \in \C_v^2,\, g_v(z,w)=0\}$ coincides with the set of points having bounded orbits.
\end{prop}

We shall also consider  the global Green function in $\C_v^3$. 
Write $\wt{P}(z_0,z_1,z_2) = z_0^dP(\frac{z_1}{z_0},\frac{z_2}{z_0})$ and $\wt{Q}(z_0,z_1,z_2) = z_0^dQ(\frac{z_1}{z_0},\frac{z_2}{z_0})$
so that $F (z_0, z_1, z_2) = (z_0 ^d, \wt{P}, \wt{Q}) $ is a homogenous map of degree $d$ that lifts $f$ to $\C_v^3$. 
Observe that in homogenous coordinates,~\eqref{eq:stand} can be rewritten as follows: 
\[
C_v^{-1}
\le
\frac{ \max\{|z_0|^d,|\wt{P}(z_0, z_1, z_2)|_v,|\wt{Q}(z_0, z_1, z_2)|_v\}}
{\max\{|z_0|_v,|z_1|_v,|z_2|_v\}^d}
\le
C_v
\]
so that the next result also holds. 
\begin{prop}\label{prop:green_global}
The function $G_v(z_0,z_1,z_2) = g_v( z_1/z_0, z_2/z_0) + \log|z_0|$
is continuous on $\C^3_v\setminus \{0\}$, $1$-homogeneous (that is, $G_v(\lambda Z) = \log\abs{\lambda}+G_v(Z)$), 
and satisfies
$G_v \circ F = d G_v$. 
We have
\[
\abs{G_v(z_0,z_1,z_2) - \log \max\{|z_0|_v,|z_1|_v,|z_2|_v\} } \le \frac{\log C_v}{d-1}
~.\]
The set $\{(z_0,z_1,z_2) \in \C_v^3,\, G_v(z_0,z_1,z_2)\le 0\}$ coincides with the set of points having bounded $F$-orbits.
\end{prop}

\subsection{Canonical heights on points}
We refer to~\cite{ACL2} for generalities on heights.
Consider the line bundle on $\mc{O}(1)\to \nP^2_\nk$. The space of sections of this line bundle
can be canonically identified with the space of linear forms  
$a_0 z_0 + a_1 z_1 + a_2 z_2$ with $a_i \in \nk$. More precisely, in the trivialization of the bundle 
over $\set{z_i\neq 0}$, this section is given by 
$\unsur{z_i} a_0 z_0 + a_1 z_1 + a_2 z_2$.  

For any $v\in M_\nk$, we consider the line bundle $\mc{O}(1)\to \nP^2_{\C_v}$ and endow it with the
metrization $|\cdot|_v$ induced by $G_v$, in the sense that 
any section  $\sigma = a_0 z_0 + a_1 z_1 + a_2 z_2$ as above 
\begin{equation}\label{eq:metriz}
|\sigma|_v ([z_0:z_1:z_2]) = \abs{a_0 z_0 + a_1 z_1 + a_2 z_2} 
  e^{-G_v(z_0, z_1, z_2)} ~
\end{equation}
(this expression is well-defined thanks to the homogeneity property of $G_v$). 
Let us now explain how this collection of metrizations defines a function on the set of algebraic points in $\nP^2$
as well as on the set of all algebraic curves in $\nP^2_\nk$ defined by an equation 
with coefficients in $\nk^{\alg}$.

For any $ p \in  (\nk^{\alg})^3\setminus\{0\}$, we set 
\[
h_f(p)
:=
\frac1{N(p)} \sum_{p' \in \Gal\cdot p} \left(\sum_{v\in M_\nk} n_v G_v(p')\right)
\]
where $\Gal$ denotes the absolute Galois group of $ \nk^{\alg}$ over $\nk$, and
$N(x)$ is the cardinality of the set $\Gal\cdot x \subset ( \nk^{\alg})^3$.

The product formula entails that 
$h_f(z_0,z_1,z_2) = h_f(\lambda z_0,\lambda z_1,\lambda z_2) $ for any $\lambda\in  \nk^{\alg}$
so that we have a well-defined function $h_f \colon \nP^2( \nk^{\alg}) \to \R$. 

\begin{prop}
The function $h_f$ takes non-negative values, and satisfies $h_f \circ f = d h_f$. 
The set $\{h_f =0\}$ coincides with the set of preperiodic points of $f$. 
Furthermore,  for any $(z,w) \in (\nk^{\alg})^2$ we have
\[
h_f(z,w)
:=
\frac1{N(z,w)} \sum_{(z',w')\in \Gal\cdot (z,w)} \left(\sum_{v\in M_\nk} n_v g_v(z',w')\right)
~.\]
\end{prop}
As above $\Gal$ denotes the absolute Galois group of $ \nk^{\alg}$ over $\nk$, and
$N(z,w)$ is the cardinality of the set $\Gal\cdot (z,w) \subset (\nk^{\alg})^2$.
The proof follows directly from Northcott's theorem, see~\cite[Corollary~1.1.1]{MR1255693}.

\subsection{Analytification of  affine curves}
Let $C$ be any irreducible algebraic curve in $\nA^2_\nk$ defined by an equation $\{R=0\}$
with $R \in \nk[z,w]$. Denote by $\overline{C}$ the Zariski closure of $C$ in $\nP^2$. 

Fix any place $v\in M_\nk$. We denote by
$C^{\an}_v$ the analytification in the sense of Berkovich of $C$ over $\C_v$. 
This is a connected, locally connected and locally compact space. When
$|\cdot|_v$ is Archimedean, hence $\C_v$ is isometric to $\C$,    
$C^{\an}_v$ is the complex
analytic subspace (possibly with singularities) defined as usual by 
the vanishing of $R$ in $\C_v^2$. 
When $|\cdot|_v$ is non-Archimedean, then 
$C^{\an}_v$ is defined as the set of multiplicative semi-norms on the ring
$\C_v[z,w]/(R)$ whose restriction to the base field equals $|\cdot|_v$.
A point is said to be rigid when the semi-norm has non-trivial kernel. 

One can also define the analytification of $\overline{C}$ by considering suitable affine coordinates in $\nP^2$
and patching the previous construction in a natural way, see~\cite[\S 3.4]{MR1070709}. Observe that $\overline{C}^{\an}_v \setminus C^{\an}_v$
consists of rigid points.

Suppose first $v$ is Archimedean. The metrization of $\mc{O}(1)$ defined by~\eqref{eq:metriz}
induces a measure $\mu_{C,v}$ on $\overline{C}^{\an}_v$ which is locally defined by 
 $\mu_{C,v} := \Delta \log \abs{\sigma}_v$ where $\sigma$ is a local section of  $\mc{O}(1)$.
  The plurisubharmonicity of $G_v$ ensures that  $\mu_{C,v}$ is a positive measure. 
 The  Lelong-Poincaré  formula  implies that 
  the mass of $\mu_{C,v}$ is equal to $\deg(C)$, and we have $\mu_{C,v} = \Delta (g_v|_{C^{\an}_v})$ 
  on $C^{\an}_v$. Observe that since $G_v$ is continuous, $\mu_{C,v}$ gives no mass to points.

  The construction is completely analogous in the non-Archimedean case.
We again obtain a positive measure $\mu_{C,v}$ on $\overline{C}^{\an}_v$ of total mass $\deg(C)$ which is given in 
$C^{\an}_v$ by  $\mu_{C,v} = \Delta (g_v|_{C^{\an}_v})$ where $\Delta$ is the Laplace operator defined by Thuillier~\cite{thuillier:phdthesis}.
Observe that the continuity of the metrization implies that $\mu_{C,v}$ does not
 charge any rigid point (but it may still charge some non-rigid point
in $C^{\an}_v$). 
We refer to~\cite[\S 1.3]{ACL2} for the details of the constructions. 

\subsection{Canonical heights on curves}
Let $C$ be any irreducible algebraic curve in $\nA^2_\nk$ as in the previous section. 
We now define the canonical height of the curve $\overline{C}$ following the recipe  given  in~\cite[\S 3.1.2]{ACL2}, 
taking $z_0$ as a section of $\mc{O}(1)$ (note that this section vanishes exactly along the line at infinity). 
We obtain: 
\begin{equation}\label{eq:01}
h_f(\ol{C}) := \sum_{p\in \overline{C} \cap L_\infty}  (\overline{C}, L_\infty)_p \times  h_f(p) + \sum_{v\in M_\nk}  n_v
\int_{C^{\an}_v}  g_v d \mu_v~,
\end{equation}
{where $ (\overline{C}, L_\infty)_p$ denotes the   intersection multiplicity of  $\overline{C}$ and $L_\infty$ at $p$}. 
Note that $h_f(\ol{C}) \ge 0$ because  the canonical height is non-negative on closed points, and the Green functions $g_v$
are also non-negative.

Define the essential minimum of $h_f$ by the following formula: 
\[
 \essmin_C(h_f) := \sup_{F \text{ finite} \subset \overline{C}(\nk^{\alg})} \inf_{\overline{C}(\nk^{\alg}) \setminus F} h_f~.
\]
\begin{thm}[Zhang's inequality~{\cite[Theorem 1.10]{zhang:smallpointsadelicmetrics}}]\label{thm:Zh}
We have
\[
2  \essmin_C(h_f)\ge \frac{h_f(\overline{C})}{\deg(C)} \ge  \essmin_C(h_f)+ \inf_{p\in \overline{C}(\nk^{\alg})} h_f(p)
~.\]
\end{thm}

Since $h_f(x)=0$ if and only if $x$ is preperiodic, we obtain:
\begin{cor}
An  irreducible algebraic curve $C$ containing infinitely many $f$-preperiodic points satisfies $h_f(\overline{C})=0$.
\end{cor}

\subsection{A first characterization of special curves}
\begin{prop}\label{prop:Green_harmonic}
Suppose that $C \subset \nA_\nk^2$ is an irreducible algebraic curve
containing a sequence of distinct points $p_n \in C(\nk^{\alg})$ such that 
$h_f(p_n) \to 0$. 

Then all points in $\ol{C} \cap L_\infty$ are preperiodic for $\ol{f}$, and for any $v\in M_\nk$ 
the function $g_v|_{C^{\an}_v}$ is harmonic on $\{g_v >0\}$.
\end{prop}

\begin{proof}
Note that $h_f(\overline{C}) \ge0$.   
{The existence of infinitely many preperiodic points on $C$  implies  that $\essmin_C(h_f) \le 0$, 
therefore from Theorem~\ref{thm:Zh} we get that 
$h_f(\overline{C}) =0$.  Since a point $p$ on $L_\infty$  is $f$-preperiodic if and only if $h_f(p)=0$, it follows from~\eqref{eq:01}
that $\int_{C^{\an}_v}  g_v d \mu_v$ for every $v\in   M_\nk$. In particular 
$g_v =0$ on the support of $\mu_v=\Delta g_v$ (recall that $g_v$ is continuous), so
 $\Delta g_v=0$ on $\{g_v>0\}$. In other words $g_v$ is harmonic there, and the result follows.} 
\end{proof}


\section{Super-stable manifolds and local estimates}\label{sec:super-stable}

\subsection{Construction of super-stable manifolds}
In this section we work under the following hypothesis:
$(\nk,|\cdot|)$
 is a complete metrized field of characteristic $0$ (which may be either Archimedean 
or non-Archimedean).  

\begin{thm}\label{thm:super-stable_manifold}
Suppose $f\colon (\nA_\nk^2, 0) \to (\nA_\nk^2,0)$ is a germ of analytic map fixing the origin of the   form
\begin{equation}\label{eq:form_f_1}
f(x,y) =  \left(\lambda x  +\mu y + g(x,y) , y^d(1+h(x,y))\right) ~,
\end{equation}
where $d\ge 2$, $\lambda\neq 0$, $\mu \in \nk$, $h(0,0)=0$, and $g(x,y) = \mathcal{O}(|(x,y)|^2)$.
Then there exists a unique smooth analytic curve which is transverse to $\{y=0\}$
and $f$-invariant.
\end{thm}

We shall call this curve the \emph{local super-stable manifold} of
the origin, and denote it by  $W^{\rm ss}_\loc(0)$. 
After a linear change of coordinates of the form $(x,y)\mapsto (x+\frac{\mu}{\lambda}y, y)$, 
we may and will assume from now on that $\mu = 0$. 
  Expressing the invariant  curve as a graph of the form 
$x=\varphi(y)$ and  making a change of coordinates of the form 
$(x,y)\mapsto (x -\varphi(y), y)$, $f$ takes the form 
\begin{equation}\label{eq:reduced_form}
f(x,y) =  (\lambda x + x\tilde g(x,y) , y^d(1+\tilde{h}(x,y))).
\end{equation}
It follows that $f$ is analytically conjugate to $y \mapsto y^d$ on  $W^{\rm ss}_\loc(0)$, hence the 
terminology.

The  result is  classical when $\nk$ is  Archimedean and/or $f$ is locally invertible
(see e.g.~\cite[Appendix]{herman-yoccoz:smalldivisornonarchi}).  
For convenience we include a proof that works  simultaneously
in the Archimedean and non-Archimedean settings, and is adapted to  $\set{y=0}$ being superattracting.

{Let us point out that  in the next section, 
we will apply this result to a polynomial mapping of $\C^2$ at a point on the line
at infinity, with   $y=0$ corresponding  to $L_\infty$, so that  $f$ will  not be polynomial in the coordinates $(x,y)$.}

\begin{proof}
As explained above, we look for an analytic map 
$y \mapsto \varphi(y)$, with $\varphi(0)=0$
such that 
the change of coordinates $\Phi(x,y) = (x + \varphi(y), y)$
satisfies
\[
\Phi^{-1}\circ f \circ \Phi (x,y) = (\lambda x + x \tilde{g}(x,y), y^d(1+\tilde{h}(x,y)))\]
with $\tilde{g}, \tilde{h}$ analytic and vanishing at $0$. 
This property is equivalent to the identities:
\[
\begin{cases}
\lambda \varphi(y)+ g(\varphi(y),y) = \varphi(y^d (1+ \tilde{h}(0,y)))
\\
\tilde{h}(0,y) = h(\varphi(y),y)
\end{cases}
\]
so that we aim at finding some analytic function $\varphi$ satisfying
\[\lambda \varphi(y)+ g(\varphi(y),y) = \varphi(y^d (1+ h(\varphi(y),y)))
~.\]
For any $r >0$, let us introduce the Banach space $\mathcal{B}_r$
that consist of those power series $\varphi(y) := \sum_{j\ge 1} a_j y^j$ which are convergent in the
disk of radius $r$, and satisfy 
$\norm{\varphi}_r := \sup_{|y|<r}  |\varphi(y)| < \infty$.
Note that in the non-Archimedean case, we have $\norm{\varphi}_r := \sup_j  |a_j| r^j$.

For any $\varphi \in \mathcal{B}_r$, we set
\[T\varphi (y) := \frac{1}{\lambda} \left( \varphi(y^d (1+ h(\varphi(y),y))) - g(\varphi(y),y) \right)~.\]
We claim that for $r>0$ and $\rho>0$ sufficiently small, 
$T$ is a well-defined strictly contracting map on $B(0,\rho)\subset \mathcal{B}_r$. 
Then, applying the   Banach fixed point theorem implies the existence of the desired $\varphi$. 

First, we may suppose that $g$ is analytic in the polydisk of radius $r$, and since $g$ vanishes up to order $2$ at the origin, we have 
\[ \abs{g(x,y)} \le C \max \{ |x|, |y|\}^2\]
for some $C>0$ and all $|x|, |y| < r$. Similarly, we may suppose that $h$  is analytic in the polydisk of radius $r>0$, 
and that $|h(x,y)| \le \frac12$ for all $|x|, |y| < r$.

Pick any  $\varphi(y) = \sum_j a_j y^j \in B(0,\rho)\subset \mathcal{B}_r$.
Reduce $r>0$ if necessary so that $\frac32 r^d <r$.
Then $\tilde{\varphi} :y\mapsto  \varphi (y^d (1+h(\varphi(y),y)))$
is well-defined and analytic on the disk of radius $r$.

In the non-Archimedean case, 
$|1+h(\varphi(y),y))| = 1$, so that 
one has
\[
\norm{\tilde{\varphi}(y)}_r = \sup_{j\ge1} \abs{a_j} r^{dj} \le r \sup_{j\ge1} \abs{a_j} r^{dj-1} \le r |\varphi|_r~.\]
In the Archimedean case, the Schwarz lemma yields
$|\varphi(y)| \le \frac{|\varphi|_r}r |y|$ for all $|y| <r$, hence 
$\norm{\tilde{\varphi}}_r \le \frac32 r^{d-1} |\varphi|_r $.
Note that we also have:
\[\norm{g(\varphi(y),y)}_r \le C \max \{\norm{\varphi}_r, r\}^2  \]
so for $\varphi\in B(0, \rho) $ we deduce  
\[|T\varphi|_r \le \frac1{\lambda}  \lrpar{\frac32 r^{d-1} \rho + C \max\set{\rho, r}^2}.\] 
By choosing $\rho = r$ and then $r$ small enough, this estimate
shows that $T\varphi$ is well-defined on the ball  $B(0,r)\subset \mathcal{B}_r$ and
$T(B(0,r)) \subset B(0,r)$. 

In order to prove that $T$ is strictly contracting, observe that
we can write  
\[
g(x,y) - g(x',y) = (x-x') \hat{g}(x,x',y)
\]
where $\hat{g}(x,x',y)$ is again analytic in the polydisk of radius $r$, and
\[
| \hat{g}(x,x',y)| \le C' \max \{|x|, |x'|, |y|\}
\]
for some constant $C'>0$.
For any pair of  analytic functions $\varphi_1, \varphi_2 \in B(0,r)\subset \mathcal{B}_r$ we infer:
\[
|g(\varphi_1(y),y) - g(\varphi_2(y),y)|
\le 
C' \norm{\varphi_1 - \varphi_2}_r \max \{\norm{\varphi_1}_r, \norm{\varphi_2}_r ,r \}
\]
hence
\[
\norm{T\varphi_1- T\varphi_2}_r
\le 
\frac1{\lambda}
(3r^d + C' r )\, \norm{\varphi_1 - \varphi_2}_r ~.
\]
Again, by choosing $r$ sufficiently small, 
we obtain  that $T$ is strictly contracting and we are done.
\end{proof}

\subsection{The rescaling argument in the repelling case}
We work in $\mathbb{A}^2_\nk$ where $\nk$ is an arbitrary complete  metrized field of characteristic 0. 
We start with a preparation lemma. 
\begin{lem}\label{lem:preparation_saddle}
Suppose $f$ is an analytic map of the form 
\begin{equation}\label{eq:form_f_2}
f(x,y) =  \left(\lambda x + x g(x,y) , y^d(1+ h(x,y))\right)
\end{equation}
where $\abs{\lambda} >1$, 
$d\ge 2$ and $g(0)=h(0)=0$. 

Then there exists an analytic  change of coordinates $\Phi$ such that
\[
\Phi^{-1} \circ f \circ \Phi (x,y)=
 (\lambda x (1 + xy \tilde{g}(x,y)) , y^d(1+x \tilde{h}(x,y)))
\]
for some analytic functions $\tilde{g},\tilde{h}$.
\end{lem}

Recall that the form~\eqref{eq:form_f_2} is what is obtained from~\eqref{eq:form_f_1} after conjugating to get $\mu = 0$ and 
declaring that the stable manifold of Theorem~\ref{thm:super-stable_manifold} is $\set{x=0}$. 
  
\begin{proof}
By B\"ottcher's theorem (see~\cite[Chapter 4]{benedetto:dyn1NA} for the non-Archimedean case)
applied to $y \mapsto f(0,y)$ we may suppose that $x$ divides $h$.
Similarly, since $|\lambda|>1$, by  
we may linearize $x \mapsto f(x,0)$, and suppose that $f$ is of the form
$f(x,y) =  (\lambda x (1+ g_1(y))+ x^2y h_1(x,y) , y^d (1+\mathcal{O}(x)))$ for some analytic functions $g_1,h_1$ with $g_1(0)=0$. 
 
We claim that there exists $\Phi(x,y) = (x(1+ \varphi(y)), y)$
with $\varphi(0)=0$ such that 
$$\Phi\inv \circ f\circ \Phi(x,y) =  (\lambda x + x^2y h_2(x,y) , y^d(1+\mathcal{O}(x)))$$ 
for some analytic function $h_2$.
Indeed, $\varphi$ is then characterized by the equation 
\[ \lambda x(1+ \varphi(y))(1+g_1(y)) + \mathcal{O}(x^2)
=
(\lambda x + \mathcal{O}(x^2)) (1+ \varphi(y^d))
\]
that is, $(1+ \varphi(y))(1+g_1(y))= (1+ \varphi(y^d))$, a solution of which is given by the infinite product
\[
1+ \varphi(y) = \prod_{k=0}^\infty \left(1+g_1\left(y^{d^k}\right)\right)^{-1}
\] and we are done.
\end{proof}

The next result is similar to \cite[Lemma 4.2]{DMM-henon}.
 
\begin{prop}\label{prop:renormalization}
Suppose $f$ is an analytic map of the form 
\begin{equation}\label{eq:form_f_4}
f(x,y) =  (\lambda x (1+ xy g(x,y)) , y^d(1+ x h(x,y)))
\end{equation}
where $|\lambda| >1$, 
$d\ge 2$ and $g,h$ are analytic functions.

Then $f^n(\frac{x}{\lambda^n} , y) \to(x,0)$ when $n\to\infty$, 
uniformly on a polydisk of sufficiently small radius
centered at the origin.
\end{prop}

\begin{proof}
Fix $0<r\leq 1/4$ small enough so that $g, h$ are both analytic on the polydisk of radius $r$ and
$|g (x,y) |, |h (x,y) | \le 1$ for   $|x|, |y| <r$. 
Let us first show that   if $\abs{x}\leq \frac{r}{2\abs{\lambda}^n}$ and $\abs{y}\leq r$, 
then the $n$ first iterates of $(x,y)$ remain in $\D_r^2$. We argue by induction. So assume that $(x_0,y_0)\in \D_{r\abs{\lambda}^{-n}/2}\times \D_r$, 
put $f^j(x_0,y_0) = (x_j, y_j)$, let $k\leq n$  and assume that $(x_j, y_j)\in \D_r^2$ for $j\leq k-1$.  
Observe that for $j\leq k-1$, $\abs{y_{j+1}}\leq 2\abs{y_j}^d$ from which it follows that  
$$\abs{y_k}\leq \lrpar{2^{1/(d-1)} \abs{y_0}}^{d^k} \leq   \lrpar{2^{1/(d-1)} r}^{d^k} \leq r\text.$$
Observe that  the first inequality together with 
 $r\leq 1/4$ also yield $\abs{y_j}\leq 2^{-d^j}$. 
For the first coordinate, we use recursively the relation 
$x_{j+1} = \lambda x_j (1+x_j y_j g(x_j, y_j))$ to get 
\begin{equation}\label{eq:xk}
\abs{x_k} \leq \abs{\lambda}^k \abs{x_0}\prod_{j=0}^{k-1} \lrpar{1+ \abs{x_j} \abs{y_j}} \leq 
\frac{r}{2} \abs{\lambda}^{k-n} \prod_{j=0}^{k-1} \lrpar{1+  2^{-d^j}} \leq r\abs{\lambda}^{k-n} ,
\end{equation}
where the last inequality follows from 
$$ \prod_{j=0}^{k-1} \lrpar{1+  2^{-d^j}} \leq \prod_{j=0}^{k-1} \lrpar{1+  2^{-2^j}} = \frac{2^{2^k}-1}{2^{2^k-1}} < 2,$$
in which the middle equality is easily obtained by induction.

Now take   $(x,y) \in \D_{r/2}$, and consider 
$f^n(\frac{x}{\lambda^n} , y)$. Denote as before $(x_0 , y_0)  = (\frac{x}{\lambda^n} , y)$ and 
$(x_j, y_j) = f^j(x_0, y_0)$.
The first part of the proof shows that $(x_j, y_j)$
is well-defined for all $j\leq n$, and that $y_n\to 0$.  Now we have 
$$x_n =  \lambda^n x_0 \prod_{j=0}^{n-1} \lrpar{1+x_jy_j h(x_j, y_j)} = x 
\prod_{j=0}^{n-1} \lrpar{1+x_jy_j h(x_j, y_j)}. $$
The inequality 
$\big|\prod (1+z_j) - 1\big| \leq \exp \lrpar{\sum \abs{z_j}} - 1$ shows that to establish the convergence $x_n\to x$ it is enough to show that $\sum_{j=0}^{n-1} \abs{x_jy_j h(x_j, y_j)}$ tends to 0.
But by~\eqref{eq:xk}, $\abs{x_j}\leq r \abs{\lambda}^{j-n}$, thus 
$$\sum_{j=0}^{n-1} \abs{x_j y_j h(x_j, y_j)} \leq \sum_{j=0}^{n-1} r \abs{\lambda}^{j-n} 2^{-d^j} \leq r\abs{\lambda}^{-n} \sum_{j=0}^\infty \abs{\lambda}^{j} 2^{-d^j},$$ and we are done.
\end{proof}

\subsection{Graph transform   for $\lambda=1$}
In this paragraph we assume that $\nk = \C$ and $f$ is of the form 
\begin{equation}\label{eq:f_parabolic} 
f(x,y) =  \left(x + g(x,y) , y^d(1+h(x,y))\right)~,
\end{equation}
with $g(x,y) = \mathcal{O}(|x,y|^2)$, $h(0,0) = 0$, and $g(x,0) = c x^{k+1} + \mathcal{O}(x^{k+2})$
 for some $k\geq 1$ and $c\neq 0$. Observe that $f\rest{\set{y=0}}$ has a parabolic point at the origin with $k$ attracting and $k$ repelling petals  (see e.g.~\cite[\S 6.5]{beardon:iterratfunction}).
An \emph{attracting petal} is a $f$-invariant (connected and simply-connected) open subset $U$ containing $0$ in its boundary, and such that, for all $z \in U$, $f^n(z) \to 0$
tangentially to some real direction (in our case, to the normalized $k$-th roots of $-c$). A \emph{repelling petal} is an attracting petal for $f^{-1}$ (they are tangent to the normalized $k$-th roots of $c$). One can chose the $k$ attracting petals and $k$ repelling petals so that their union fills up a punctured neighborhood of the origin.

A vertical graph $V$ in a domain of the form $\Omega\times \D_\rho$
 is a  
 submanifold of the 
form $V:= \set{(\varphi(y), y), \ y\in  \D_\rho}$ for some holomorphic function 
$\varphi : \D_\rho \to \Omega$.
In the next theorem we consider 
pull backs of such graphs in $\D_r^2$ in 
 the graph transform sense, that is, 
when pulling back some vertical graph under $f$, we keep only the component of $f\inv (V)\cap \D_r^2$
containing $f\inv(V\cap \set{y=0})$. Abusing notation we simply denote it by $f\inv (V)$. 

\begin{thm}\label{thm:graph_transform_parabolic} 
Let $f\colon(\C^2, 0) \to (\C^2, 0)$ 
  be of the form~\eqref{eq:f_parabolic} 
  with $h(0)=0$, $g(x,0)\not\equiv 0$ and $g(x,y) = \mathcal{O}(|(x,y)|^2)$. 
Let $U$ be any repelling petal of   
$f\rest{\set{y=0}}$ and consider a germ $V$ of analytic curve transverse 
to $\set{y=0}$, intersecting $U$.
 
Then there exists $r>0$  depending only on $f$, such that  for  large enough $n$, 
the analytic sets $f^{-n}(V)$ are   vertical graphs in $\D_r^2$  
converging  to $W^{\rm ss}_\loc(0)$ in the $C^1$ topology. \end{thm}

 \begin{lem}\label{lem:preparation_parabolic}
Suppose $f$ is a holomorphic map of the form~\eqref{eq:f_parabolic} as in Theorem~\ref{thm:graph_transform_parabolic}.
Then there exist an integer $k\ge1$, and an analytic change of coordinates $\Phi$ such that
\[
\Phi^{-1} \circ f \circ \Phi (x,y)=
 (x + x^{k+1}+ x^{2k+1} \tilde{g}(x,y), y^d(1+x \tilde{h}(x,y)))
\]
for some analytic functions $\tilde{g}, \tilde{h}$.
\end{lem}

\begin{proof}
The proof is essentially contained in~\cite[Proposition 2.3]{hakim}. We provide the details for the sake of completeness. 
By Theorem~\ref{thm:super-stable_manifold} and by applying the Böttcher 
  theorem to $y\mapsto f(0,y)$ we may assume that both $g$ and $h$ are divisible by $x$, 
so that  we may write
\begin{equation}\label{eq:f_parabolic2} 
f(x,y) =  \left(x + x g(x,y) , y^d(1+x h(x,y))\right) ~,
\end{equation}
with $g(0,0)=0$.

By a local change of coordinates involving only  $x$, we can arrange so that 
$f(x, 0) = (x +  x^{k+1} + \mathcal{O}(x^{2k+1}), 0)$.  
Expand the first coordinate of $f$ in power series of $x$ as follows:
$$x \circ f(x,y) =  x (1+g_0(y)) + x^{k+1} (1+ g_k(y))+ \sum_{j\neq 0,k}^\infty x^{j+1}g_j(y),$$ with $g_j(0)  = 0$ for $0 \leq j \leq 2k-1$.  

We claim that  for all $n \leq 2k-1$, we can conjugate $f$ by a germ of invertible holomorphic  map such that $g_j \equiv 0$ for every $j\leq n$. Applied to  $n=2k-1$, this claim implies the proposition.

For $n=0$ this is done by a change of coordinates of the 
form $\Phi_0(x,y):= (x(1+\varphi_0(y)), y)$ such that 
$(1+g_0) (1+\varphi_0) = 1+\varphi_0(y^d)$. This equation can be solved exactly 
as  in Lemma~\ref{lem:preparation_saddle}. 

Now assume that $n>0$, and that the result has been achieved up to $j = n-1$.
Put $\Phi_n(x,y) = \big(x(1+ \varphi_n(y) x^{n}), y\big)$, so that $\Phi_n\inv(x,y) = \big(x(1 - \varphi_n(y)x^{n} + \mathcal{O}(x^{n+1})), y\big)$.
Depending on the position of $n$ and $k$, we obtain: 
\[
x \circ ( \Phi_n\inv  \circ f \circ \Phi_n)=
\begin{cases}
x + x^{n+1} \big(\varphi_n(y)+g_n(y) - \varphi_n(y^d)\big)+  \mathcal{O}(x^{n+2})
& \text{ if } n <k,
\\
x + x^{n+1} \big(\varphi_k(y)+1+g_n(y) - \varphi_n(y^d)\big) +  \mathcal{O}(x^{n+2})
& \text{ if }  n=k,
\\
x + x^{k+1} + x^{n+1} \big(\varphi_n(y)+g_n(y) - \varphi_n(y^d)\big) + \mathcal{O}(x^{n+2})
& \text{ if }  n>k.
\end{cases} 
\]
Therefore, to render the term in $x^{n+1}$ constant, it is 
enough to solve the equation $-g_n(y) = \varphi_n(y) - \varphi_n(y^d)$ which can be done by setting 
$\varphi_n(y) =  - \sum_{m=0}^\infty g_n\big(y^{d^m}\big)$. 
\end{proof}

\begin{proof}[Proof of Theorem~\ref{thm:graph_transform_parabolic}]
By the previous lemma, we may suppose that 
\[f(x,y) = \left(x+x^{k+1}+ x^{2k+1} g(x,y), y^d (1+ x h(x,y))\right)\] with $g$ and $h$ holomorphic near the origin. 
Note that 
$f\rest{y=0}$ has a repelling petal along the positive real axis. 
Fix $r>0$ such that
$f$ is holomorphic and injective on a neighborhood of $\overline\D_r^2$. 
 Reducing and rotating the petal if necessary, we may assume that 
\[U:=\set{x:\ \arg(x)\in \lrpar{-\frac{\pi}{4k}, \frac{\pi}{4k}} , \ \abs{x}<r}.\] 
The holomorphic map $ z =   \lrpar{kx^k}\inv$ is univalent on 
$U\times \D_r$, and  takes its values in 
 \[\Omega_R := \set{z: \ \arg(z)\in  \lrpar{-\frac{\pi}{4}, \frac{\pi}{4}}, \ \abs{z}>R},\] where 
 $R=(kr^k)^{-1}$.
  The expression of 
 $f$ in the coordinates $(z,y)$ is of the form
\begin{equation}\label{eq:f_parabolic3}
f(z,y)  = \lrpar{z-1+ \unsur{z} a(z,y),  y^d \lrpar{1+  \unsur{z^{1/k}} b(z,y)}}. 
\end{equation}
so that $f$ is now defined in $\Omega_R\times \D_r$. 
 Fix    $M>0$ such that 
 \begin{equation}\label{eq:estimates}
 \abs{a}, \ \abs{b}, \abs{\frac{\fr a}{\fr z}}, \abs{\frac{\fr a}{\fr y}}, \ \abs{\frac{\fr b}{\fr z}}, \ \abs{\frac{\fr b}{\fr y}}
 \leq M \text{ on }\Omega_R\times \D_r, 
 \end{equation} 
and reduce $r$ if necessary so that $r<\frac1{10 d}$ and $MR^{-1/k}\leq \frac{1}{100}$.

 For any $\rho<r$ and $\sigma >0$ we let 
 \[\cG(\rho, \sigma) =\{\varphi \colon \D_\rho \to \Omega_R \text{ holomorphic  s.t. } \sup_{\D_\rho} 
 \abs{\varphi'} \le \sigma\}~.\]
\begin{lem}\label{lem:vertest}
Suppose that $\sigma \rho < \frac1{100 d}$.  
 For any vertical graph $\Gamma$ 
determined by $\varphi\in \cG(\rho, \sigma)$, then $f\inv \Gamma$ is 
a vertical graph determined by a function   $\psi \in \cG(\rho_1, \frac1{10})$
where
$\rho_1  = \min\lrpar{ (\rho/2)^{1/d}, r}$ and  $\Re(\psi) \geq \Re(\varphi(0)) + 9/10$.
\end{lem}
 
Assuming this result for the moment, 
 let us  conclude the proof of the theorem.  
Let  $V$ be any germ of curve intersecting transversely   $\set{y=0}$ 
 at $(x_0, 0)\in U$. Then $V$ is a graph of slope $\sigma$ over some  
  disk $\D_\rho$ in the second coordinate, for some
 $\rho>0$. Reduce $\rho$ by trimming $V$ if necessary so that $\sigma\rho< 1/100d$.
Since $r/10 < 1/100d$, the  previous lemma implies that 
 we can define inductively a sequence of vertical graphs
$V = V_0$,  $V_n :=  f\inv V_{n-1}$, where $V_n$ is defined by a holomorphic function $z = \varphi_n(y)$ of uniformly bounded slope 
over $\D_{\rho_n}$, and furthermore  $\rho_n = r$ for  large enough   $n$.   
Moreover, we have $\Re(\varphi_n) \geq \Re(\varphi(0)) + 9n/10$, hence, coming back to the 
$(x,y)$ coordinates, we see that  $V_n = \{ x= (k\varphi_n(y))^{-1/k}\}$
converges in the $C^1$-topology to the curve $W^{\rm ss}_\loc(0)= \{x=0\}$.
\end{proof}

\begin{proof}[Proof of Lemma \ref{lem:vertest}]
Let $\Gamma$ be a vertical graph of equation $z=\varphi(y)$ in the $(z,y)$ coordinates, 
with $\varphi\in  \cG(\rho, \sigma)$. Then 
the equation of  $f\inv \Gamma$ is given by  $z= \ell(z,y)$, where 
\[ \ell(z,y) =  \varphi\lrpar{y^d \lrpar{1+  \unsur{z^{1/k}} b(z,y)}} +1  -  \unsur{z} a(z,y). \] 
Fix $y_0\in \D_{\rho_1}$. We 
  show that the equation $ z = \ell(z,y_0)$ admits a unique solution $z\in \Omega_R$. First, observe that by the 
  estimates~\eqref{eq:estimates} on $\abs{a}$ and $\abs{b}$, for $z\in \Omega_R$  we have   
\begin{align}
\label{eq:translation}\abs{\ell(z, y_0 )  - (\varphi(0)+1)}&\leq \abs{\varphi\lrpar{y_0^d \lrpar{1+  \unsur{z^{1/k}} b(z,y)}}  - \varphi(0)} + \frac{M}{R}
\\ \notag& \leq 2 \rho_1^d  \sigma +\frac{1}{100} \leq \rho  \sigma+\frac{1}{100} \leq 
 \unsur{10}\ ,
 \end{align}
hence  $\ell(\cdot, y_0)$ maps   $\Omega_R$ to itself. Next, we see that 
\begin{align*} \abs{\frac{\fr \ell}{\fr z}(z,y_0)} 
 &\leq \abs{y_0^d\lrpar{\unsur{z^{1/k}}  {\frac{\fr b}{\fr z}} - \frac{1}{kz^{1/k+1}}b}} \cdot \abs{\varphi' \lrpar{y_0^d \lrpar{1+  \unsur{z^{1/k}} b}}} + \abs{\frac{a}{z^2}- \unsur{z} {\frac{\fr a}{\fr z}}}\\
 &\leq \rho_1^d \frac{2M}{R^{1/k}}  \sigma +  \frac{2M}{R}\leq \frac{4}{100}\rho  \sigma + \frac{2}{100}\leq \frac{1}{10}\ ,
 \end{align*}
so $\ell(\cdot, y_0)$ is a contraction and the equation 
  $z = \ell(z, y_0)$ has a unique solution. This means that $f\inv\Gamma$ is a vertical graph over $\D_{\rho_1}$ determined by a holomorphic
  function $\psi$ satisfying $\psi(y) = \ell(\psi(y),y)$. 
  The slope of this graph can  be estimated as above: 
 \begin{align*} 
 \abs{\psi'(y)} 
 & \leq \abs{\frac{ {\fr\ell}/{ \fr y} }{1 -  {\fr\ell}/{ \fr z}}}
\le \frac{10}{9}  \left(  
\sigma \left(d \rho^{d-1} \left(1+\frac{M}{R^{1/k}}\right) + \rho^d \frac{M}{R^{1/k}}\right)
+ \frac{M}R\right) \le \frac1{10}.
 \end{align*}  
 Finally, the estimate $\Re(\psi) \geq \Re(\varphi(0)) + 9/10$ follows from~\eqref{eq:translation} and we are done. 
\end{proof}

\section{Proof of Theorem~\ref{thm:DMM} when $\nk$ is a number field}\label{sec:conclusion_number_field}

Here we establish the following more precise form of our main theorem, in the number field case.

\begin{thm}\label{thm:main_precised}
Let $\nk$ be a number field and  $f$ be  a regular polynomial map 
of $\A^2_\nk$. Denote by  $h_f$  the induced    canonical height.

Suppose that $C \subset \nA_\nk^2$ is an irreducible algebraic curve
containing a sequence of distinct points $p_n \in C(\nk^{\alg})$
such that 
$h_f(p_n) \to 0$. If there exists a point of $\overline C\cap L_\infty$ which is not 
eventually superattracting, then $C$ is preperiodic. 
\end{thm}

Let $f_\infty$ be the restriction  to the line at infinity $L_\infty$ of the extension of $f$ to $\P^2$. 
By Proposition~\ref{prop:Green_harmonic}, all points in  $\overline C\cap L_\infty$ are preperiodic, 
so   we may replace $f$ by $f^N$ and $C$ by 
$f^N(C)$, to assume that  $\overline C\cap L_\infty$ contains a fixed point $p$
which is not super-attracting. 
 Let  $\lambda = f_\infty'(p)$ be the multiplier of $p$ along $L_\infty$. 
 Then one of the two following mutually disjoint cases occur:
 \begin{enumerate}[(a)]
 \item either $\lambda$ is a root of unity;
 \item or there is a place $v\in M_\nk$ such that $\abs{\lambda}_v >1$. 
 \end{enumerate}
In the remainder of this section we  split the proof of the theorem according to these two cases.

\subsection{When $\lambda$ is a root of unity} 

In this situation we iterate $f$ further so that $\lambda=1$, and work over the complex numbers. Since $p$ belongs to the Julia set $J(f_\infty)$, which is a perfect set, the union of attracting and repelling petals cover a punctured neighborhood of $p$, 
and the attracting petals of $p$ are contained in the Fatou set, 
we see that there is a repelling periodic point $q$
of $f_\infty$ contained in some local repelling petal of $p$. 
Then the local (super-)stable manifold of $q$ is a
disk transverse to $L_\infty$ at $q$, and by Theorem~\ref{thm:graph_transform_parabolic}, the local 
truncated pull-backs    $f^{-n}(W^{\rm ss}_\loc(q))$ 
under $f^n$ converge to $W^{\rm ss}_{\loc}(p)$ when $n\to \infty$. (Here $f\inv$ refers to the branch fixing $p$. )

Assume by way of contradiction that  
$\overline C$ does not locally coincide with $W^{\rm ss}_{\loc}(p)$. We rely on the following claim.
 \begin{lem}\label{lem:trans}
 {For all $n$ large enough,  $C$ intersects $f^{-n} (W^{\rm ss}_\loc(q))$ transversely.}
 \end{lem}
 {Let us work in a sufficiently small neighborhood $V$ of $p$,
 so that the critical set of $f$ in $V$  is equal to the line at infinity. Fix a repelling periodic point 
 $q\in V\cap L_\infty$  such that the backward orbit of $q$ under the branch of $f\inv_\infty$ fixing $p$ is contained in $V\cap L_\infty$. Let $k$ be its period. 
 For large $n$, by the previous lemma there exists  a transverse intersection point 
 $p_n$ of $C$ and $f^{-n} (W^{\rm ss}_\loc(q))$ in $V\cap \C^2$. Fix a small disk 
 $\Delta$ in $C$ containing this intersection point  
 Since $V\cap \C^2$ is disjoint from the critical set, 
  for every $\ell\ge 0$, $f^{n+\ell k}(p_n)$ is a transversal intersection of $C$ and $W^{\rm ss}_\loc(q)$, and by the Inclination Lemma,  the derivative of 
$f^{n+ \ell k}$ in the direction of $\Delta$ tends to infinity with $\ell$ , thus 
 $(f^{n+\ell k}\rest{\Delta})_{\ell\geq 0}$ is not a normal family. On the other hand, 
 by Proposition~\ref{prop:Green_harmonic}, $g\rest{\Delta}$ is harmonic, which 
 implies that $(f^{n+\ell k}\rest{\Delta})$ is  normal  (see the proof of~\cite[Prop. 5.10]{fornaess-sibony:cplxdynhigherdim2}). 
 This contradiction shows that 
 $\overline C$   locally coincides with   $W^{\rm ss}_{\loc}(p)$ near $p$,  so it is fixed under $f$, and by irreducibility this property propagates to the whole of $C$. This completes the proof in this case.   \qed
}
\begin{proof}[Proof of Lemma~\ref{lem:trans}]
{
The proof is an adaptation of that of~\cite[Lemma~6.4]{zbMATH00447094} which treats the case where the germ of $C$ at $p$ is smooth. 
We work in a small neighborhood of $p$ with coordinates $(x,y)$ in the unit bidisk, 
such that the line at infinity is equal to $\set{x=0}$, 
and $W^{\rm ss}_\loc(p)=\set{y=0}$. 
By Theorem~\ref{thm:graph_transform_parabolic}, in this  bidisk
$D_n= f^{-n} (W^{\rm ss}_\loc(q))$ is a graph of the form  $y=f_n(x)$, 
for some holomorphic map $f_n:\D\to \D$, and 
 $(f_n)$ converges uniformly to 0.  
Since $D_n$ does not intersect $\set{y=0}$, it follows that
$f_n$ does not vanish on $\D$. Write $y_n = f_n(0)$. Then by the Harnack inequality applied to 
the negative harmonic function $\log\abs{f_n}$, for $x\in \D$
we get 
\begin{equation}\label{eq:harnack}
\abs{y_n}^{\frac{1+\abs{x}}{1-\abs{x}}} \leq  \abs{f_n(x)}\leq \abs{y_n}^{\frac{1-\abs{x}}{1+\abs{x}}}.
\end{equation}
}

{
Up to scaling in one direction, we may choose a local injective parameterisation 
$t \mapsto (t^p(1+ h(t)), t^q)$ of $C$ with $p, q$ coprime integers and $h(0)=0$. 
The points of intersection of $C$ and $D_n$ are determined by the equation 
\begin{equation} \label{eq:solve}
t^q-y_n= g_n\left(t^p(1+ h(t))\right)~,
\end{equation}
with $g_n(\cdot)= f_n(\cdot) -y_n $. 
Since  $C$ and $\set{y=0}$  have an isolated intersection of multiplicity $q$ at $0$,
by the persistence of proper intersections 
  we know that for large $n$  this equation admits $q$ solutions in any given neighborhood of 0. 
To prove the lemma it is enough to show that these solutions are simple.  More precisely, we will 
prove that for 
  large $n$, the equation~\eqref{eq:solve} admits $q$ simple solutions of the form 
$c_n + \delta$ with $c_n^q=y_n$ and $|\delta|\ll   |y_n|^{1/q}$.}

Fix $\eta< p/q$. We know that for large enough $n$,  all intersection points 
are contained in {$D\lrpar{0, \frac{\eta}{2-\eta}}$}, so that 
$\frac{1-\abs{x}}{1+\abs{x}}\geq 1 - \eta$.  
By the Schwarz lemma and \eqref{eq:harnack}, in this domain we have  
\[\abs{g_n(x)} \leq 2 \norm{f_n}_{D\lrpar{0, \frac{\eta}{2-\eta}}} \abs{x} \leq 2\abs{y_n}^{1-\eta}\abs{x}.\]
Let $c_n$ be one of the $q^{\rm th}$ roots of $y_n$  and write $t = c_n+\delta$. 
Then~\eqref{eq:solve} becomes  
\begin{equation*} 
(c_n+\delta)^q - c_n^q = g_n( (c_n+\delta)^p(1+h(c_n+\delta)))
\end{equation*}
We can write $(c_n+\delta)^q - c_n^q= \delta \theta_n(\delta)$, with 
  $|\theta_n(\delta)|\ge A  |c_n|^{q-1}$ when $|\delta|\ll |c_n|$ and the previous equation becomes 
\[
\delta = g_n( (c_n+\delta)^p(1+h(c_n+\delta))) \theta_n(\delta)^{-1}=: G_n(\delta)~.
\]
The previous estimates give $|G_n(\delta)| \lesssim \abs{y_n}^{1-\eta} \abs{c_n}^p |c_n|^{1-q} = O( 
\abs{c_n}^{1+p - q\eta})$. The  choice of $\eta$ guarantees that 
$1+p-q\eta>1$, hence  Rouché's theorem gives a unique solution to our equation in the range $\abs{\delta}\ll\abs{c_n}$, 
 as required.
\end{proof}

\begin{rmk}
The same proof works   when $\abs{\lambda}_v>1$ at some Archime\-dean place, 
which may {help the reader} 
understand the non-Archimedean argument below.
\end{rmk}

\subsection{When $\abs{\lambda}_v >1$}
We may assume that $p=[0:0:1]$, and by Lemma~\ref{lem:preparation_saddle} 
find a local analytic isomorphism 
$(x,y) \mapsto \psi(x,y) = [z_0(x,y):z_1(x,y):1]$
such that $\psi(0) = p$, $z_0(x,0) = 0$ so that
$\set{y=0}$ corresponds to  the line at infinity\footnote{Beware that coordinates are swapped here
: $\set{z_0 = 0}$ corresponds to $\set{y=0}$.}, 
and write
\[
\wt{f}:=
\psi^{-1}\circ f \circ \psi: (x,y) \longmapsto   \Big(\lambda x \big(1+ xy g(x,y)\big) , y^d\big(1+ x h(x,y)\big)\Big)
~.
\]
If $\overline{C}$ has an analytic branch at $p$ which coincides with $\set{x=0}$
then  $C$ is fixed as above, and we are done. Otherwise, we may find 
a Puiseux parameterization of a branch of $\overline{C}$ at $p$ in the $(x,y)$ coordinates
of the form $\Gamma(t)= (t^q , \gamma(t))$, where $\gamma$ is analytic and 
defined in a small disk $\D_\delta$,   
and $q\in \N^*$. We seek a contradiction. 

We lift $\psi$ to $\A^3_\nk$, and set $\Psi(x,y) = (z_0(x,y),z_1(x,y),1)$. 
As in \S\ref{sec:local green}, we lift $f$ to a homogeneous polynomial map $F\colon \A^3_\nk \to  \A^3_\nk$, of the form 
$F(z_0,z_1,z_2) = (z_0^d ,\wt{P},\wt{Q})$. Since $\Psi \circ \psi \inv$ is a local section of the projection $\A^3\setminus \set{0}\to \P^2$, 
$\Psi \circ \widetilde{f}$   
must be a multiple of $F\circ \Psi$. From the expression of $F$ we obtain
\[
\Psi \circ \widetilde{f} = \frac{F \circ \Psi}{\wt{Q} \circ \Psi}~.
\]
To simplify notation, we write 
$F^n(z_0,z_1,z_2) = (z_0^{d^n} ,\wt{P}_n,\wt{Q}_n)$. 
We consider the $1$-homoge\-neous Green function $G_f \colon \A^{3,\an}_v \to \R$ of Proposition~\ref{prop:green_global};  
it satisfies
$G_f \circ F = d G_f$ and $g_f(z_1,z_2)= G_f(1,z_1,z_2)$.

Observe that  $h := G_f \circ \Psi \circ \Gamma $ is a continuous  function on $\D_\delta$. 
Since  $h(t) = g_f \circ \psi \circ \Gamma (t) + \log |z_0 \circ \Gamma(t)|$, 
by Proposition~\ref{prop:Green_harmonic}, $h$ is harmonic on $\set{t\neq0}$. 
Since it is continuous at $0$, it is also harmonic on $\D_\delta$, see, e.g.,~\cite[Lemma 3.7]{favre-gauthier:classcurvescubicpol}.
Write
\begin{align*}
d^{nq} h (t) 
&= 
G_f \circ F^{nq}\circ \Psi \circ \Gamma (t)
\\
&= 
G_f \circ \Psi \circ \widetilde{f}^{nq} \circ \Gamma(t) + \log \abs{\wt{Q}_{nq} \circ \Psi \circ \Gamma(t)}~.
\end{align*}
By Proposition~\ref{prop:renormalization}, $\widetilde{f}^{n} \big(\frac{x}{\lambda^n},y\big) \to (x,0)$ uniformly in a neighborhood of the origin, hence
\[G_f \circ \Psi \circ \widetilde{f}^{nq} \circ \Gamma \big(\tfrac{t}{\lambda^{n/q}}\big) \to G_f \big(0, z_1(t^q,0), 1\big)\]
as $n\to\infty$. 
On the other hand, 
since $d^{nq} h \big(\tfrac{t}{\lambda^{n/q}}\big)  - \log \abs{\wt{Q}_{nq} \circ \Psi \circ \Gamma\big(\tfrac{t}{\lambda^{n/q}}\big)}$ is a sequence of harmonic functions, 
it follows that $t \mapsto G_f (0, z_1(t^q,0), 1)$ is   harmonic as well. 

Now, observe that the restriction of $f$ to the line at infinity is 
$f_\infty [z_1:z_2] = \big[ \wt{P}(0,z_1,z_2) : \wt{Q}(0,z_1,z_2)\big]$, so that
$G_f (0,z_1, z_2)$ is the global Green function of $f_\infty$. 
The equilibrium measure of $f_\infty$ is the probability measure
on the analytification of $L_\infty$ defined by  
$\mu_{f_\infty} := \Delta G_f (0,z_1, 1)$ in the chart $z_2\neq 0$. 
Its support is the Julia set of $f_\infty$ (see~\cite[Theorem~13.39]{benedetto:dyn1NA}), 
and it contains all repelling (rigid) fixed points, see~\cite[Theorem~8.7]{benedetto:dyn1NA}.
Therefore,  $z_1\mapsto 
G_f (0, z_1, 1)$ cannot be harmonic near $0$,
hence the function $t \mapsto  G_f (0, z_1(t^q,0), 1)$ cannot be harmonic either. 
This  contradiction concludes the proof. \qed
 
\begin{rmk}\label{rmk:fixed}
Under the assumptions of Theorem~\ref{thm:main_precised}, the proof shows that the 
preperiod $k$ of $F$ and the period of $f^k(C)$ are exactly the same as that of any of its non-superattracting points at infinity. 
\end{rmk}

\section{Proof of Theorem~\ref{thm:DMM} for arbitrary $\nk$}\label{sec:spec}

In this section we use a specialization argument to deal with maps defined over arbitrary fields. 
It shares some arguments with~\cite[\S 5]{DMM-henon} (see also \cite[\S 7]{finite_orbits}). 
Nevertheless,  new ideas are needed to  deal with preperiodic points  instead of periodic ones. 

We are in the setting of Theorem~\ref{thm:DMM}, so we assume that $f$ is a regular polynomial map of $\A^2$ of degree $d\ge 2$ defined over a field $\nk$
of characteristic $0$, and $C$ is a curve containing
 an infinite set $\mathcal P = \set{p_n, \ n\geq 0}$ 
 of preperiodic  points 
and whose closure $\overline{C} \cap L_\infty$ contains at least one point which is not
eventually super-attracting.   

By enlarging  $\nk$ if necessary we may assume that it contains the algebraic closure $\Q^{\rm alg}$
of its prime field. 
Let $R$ be the sub-$\Q^{\alg}$-algebra of $\nk$ generated by 
all coefficients defining $f$ and $C$. Its fraction field $K$ 
is  finitely generated over $\Q^{\alg}$. Let $S = \spec R$.
This is an affine variety defined over $\Q^{\alg}$, and
elements of $R$ can be seen as regular functions on $S$. 

Inverting some elements of $R$ if necessary, we may suppose that 
$C$ is flat over $S$, and $f$ extends as a morphism $f\colon \P^2_S   \to  \P^2_S$.
We let $\pi \colon \P^2_S \to S$ be the canonical projection,  and write $\P^2_s = \pi^{-1}(s)$. 
We also let $\A^2_s:= \A^2_S \cap  \pi^{-1}(s)$.

For each  (scheme theoretic) point $s\in S$, we write $C_s = C\cap \A^2_s$ and let $\overline{C}_s$
  be the closure of $C_s$ in $\P^2_s$. The flatness of the morphism $C\to S$ implies 
  $\overline{C}_s$ (hence $C_s$) to be a curve. 
  Similarly, we let $f_s \colon \P^2_s \to \P^2_s$ be the induced map on the fibers: 
this is an endomorphism of degree $d$.

We also denote by $p_{n,s}\in \A^2_s$ the specialization of $p_n$. Note that 
 $p_n$ is defined over some finite extension of $K$ which depends on $n$.

The first result does not use the assumption that our infinite set of preperiodic points lies on a curve.

\begin{prop}\label{prop:keeping}
Let as above $ {f}\colon \P^2_S \to \P^2_S$  be a family of endomorphisms   over an affine variety defined over $\Q^{\alg}$, and let 
$\mathcal {P}  = \set{p_n, \ n\geq 0}$ be an infinite family of preperiodic points. Then there exists a 
non-empty Zariski open and dense subset $U\subset S$ such that for any $s\in U\cap S(\Q^{\mathrm{alg}})$, 
$\mathcal P_s = \set{p_{n, s}, \ n\geq 0}\subset \A^2_s$ is infinite.
\end{prop}

Before starting the proof, let us fix some additional notation. For each $n\geq 0$, 
 we denote by $k_n$ the preperiod of $p_n$, so that 
 $q_n  := f^{k_n}(p_n)$ is the first periodic point in the orbit of $p_n$. 
We let $\ell_n$ be the (primitive) period of $q_n$. 

\begin{proof}
We may suppose that there exists a parameter $s_0\in S(\Q^{\rm alg})$ such that $\mathcal P_{s_0}$ is finite (otherwise we take $U=S$ and the proof is complete). 

\begin{lem}
The family of periodic points $(q_n)$ is finite. 
\end{lem}

\begin{proof}
We follow the arguments of~\cite[\S 5]{DMM-henon}.
Set $\mathcal Q = \set{q_n, \ n\geq 0}$. 
For each $\ell\geq 1$, we consider the subvariety $\mathrm{Per}_\ell$ 
of $\P^2_S$ defined by the equation 
$f^\ell(z) = z$. 
Since $\P^2_S\to S$ is proper, the structure map $\mathrm{Per}_\ell \to S$ is also proper. 

Let $\mathcal Q_\ell$ be the union of the irreducible components of $\mathrm{Per}_\ell$ containing a point of 
$\mathcal P$. Its underlying set is  the Zariski closure of $\mathcal P\cap \mathrm{Per}_\ell$, hence
$\mathcal Q_\ell \to S$ is proper.
Observe that for $x\in \mathcal Q_{\ell, s}$, the multiplicity of $x$ as a point of 
$\mathcal Q_{\ell, s}$ equals 
 its multiplicity as a fixed point of $f_s^\ell$. 
By Nakayama's lemma and the properness of $\mathcal Q_\ell$ over $S$,
the function 
\begin{equation}\label{eq:mult}
 s\longmapsto \sum_{x\in \mathcal Q_{\ell, s}} \mathrm{mult}_x(\mathcal Q_{\ell, s})
\end{equation}
is   upper semi-continuous for the Zariski topology, hence 
\begin{equation}\label{eq:mult2}
\sum_{q\in \mathcal Q_\ell} \mathrm{mult}_q(\mathcal Q_\ell)
\leq
\sum_{x\in \mathcal Q_{\ell, s_0}} \mathrm{mult}_x(\mathcal Q_{\ell, s_0}),
\end{equation}
where the left hand side is the value of~\eqref{eq:mult} at the generic point.  
By assumption $\mathcal P_{s_0}$ is a finite set, hence so does 
$\mathcal Q_{s_0} = \set{q_{1, s_0}, \ldots , q_{r, s_0}}$ 
and by the Shub-Sullivan theorem~\cite{shub-sullivan}, there exists a uniform bound $C>0$
such that for every $j$, and for any $\ell$, we have
\[\mathrm{mult}_{q_{j,s_0} }(\mathcal Q_{\ell, s_0}) \leq \mathrm{mult}_{q_{j, s_0}}(\mathrm{Per}_{\ell, s_0}) \le C~.\]
It then follows from~\eqref{eq:mult2} that 
\[
\# \mathcal Q_\ell \le \sum_{q\in \mathcal Q_\ell} \mathrm{mult}_q(\mathcal Q_\ell)
\le r C\]
hence $\bigcup_\ell \mathcal Q_\ell$  is finite, 
   as was to be shown. 
\end{proof}

By  the previous lemma, 
replacing $f$ by some iterate 
$f^N$ we may assume that all periodic points $q_n$ are fixed.
Since $\mathcal P$ is infinite, one of these fixed points, say  $q_1$,
 admits   infinitely many preimages   in $\mathcal P$. 
 We may denote $q=q_1$ and suppose $\mathcal P$ is made of an infinite set of 
 preimages of $q$, that is, (after possible reordering of $\mathcal P$) for any $p_n\in \mathcal P$ there is a minimal 
 $k_n\geq 0$ such that $f^{k_n}(p_n) = q$, and that $k_{n+1}>k_n$.
We may adjoin to $R$ the coordinates of $q$ so that 
$q \in \A^2(R)$, i.e., for any  $s\in S$, $q_s$ is a single point
(to say it differently, we replace $S$ by a  branched cover of a Zariski open dense subset of $S$).
 
 Let $d(s)$ be the local degree of $f_s$ at $q_s$, which is 
 upper semicontinuous for the Zariski topology. 
 Since $f_s$ is a finite map of degree $d^2$, $d(s)\leq d^2$ for every $s$. 
Thus there is an analytic hypersurface $H$ such that $d(s)  = d_{\mathrm{min}}$ is constant for 
$s\in S\setminus H$.

We claim that $\mathcal P_{s_1}$ is infinite
for any $s_1\in S\setminus H$. We argue again 
in the complex analytic category fixing an embedding
 $\Q^{\alg}$ into $\C$.  Observe that for any point $s\in S(\C)$,    $p_{n,s}$ is  
  a finite set included in the fiber $\A^{2,\an}_s\simeq \C^2$ (not necessarily reduced to a single point since 
   $p_n$ lies in a finite extension of $R$).

 Fix an analytic neighborhood $V$ 
of $q_{s_1}$ in $\P^{2,\an}_{s_1}(\C)$ such 
that $f_{s_1}\inv(q_{s_1})\cap \overline V = \set{q_{s_1}}$. 
Since $d(s)$ is locally constant near $s_1$, 
there is an analytic neighborhood $W$ of $s_1$ in $S^{\an}(\C)$ such that 
for $s\in W$,   
\begin{equation}\label{eq:isolated}
f_{s}\inv(q_{s})\cap   V = \set{q_{s}}
\end{equation}
Choose any $n>m$, and suppose by contradiction that $p_{m, s_1} = p_{n, s_1}$. 
Since $k_n-1\geq k_m$, we have
$$ f^{k_n-1}_{s_1}(p_{n, s_1})= f^{k_n-1}_{s_1}(p_{m, s_1}) = q_{s_1}.$$ 
Thus, for $s$ close to $s_1$, the finite set  $f^{k_n-1}_{s_1}(p_{n, s_1})$ is contained in 
  $V$, hence by~\eqref{eq:isolated}, $f^{k_n-1}_{s}(p_{n, s}) =q_s$, and by analytic continuation this property holds throughout $S$, which contradicts the definition of $k_n$.  

This shows that the  $p_{n,s}$ are all distinct for all $s\in S\setminus H$, and
concludes the proof of Proposition~\ref{prop:keeping}. 
\end{proof}

\begin{prop}\label{prop:infinity_periodic}
Let   ${f}\colon \A^2_\nk \to \A^2_\nk$  be a regular polynomial map and $C\subset \A^2_\nk$ be an algebraic curve containing infinitely many preperiodic points. 
Then every point of $\overline C\cap L_\infty$ is preperiodic under $f|_{L_\infty}$. 
\end{prop}

\begin{proof}
We keep the same  formalism and notation 
 as above, so that $f$ is viewed as a family over $S$. Write 
 $\overline C\cap L_\infty = \set{c_1,\ldots, c_r}$  and without loss of generality enlarge $R$ so that 
 the points at infinity $c_i$ have their coordinates in $R$. Fix  $i\in \set{1, \ldots, r}$ for the remainder of the proof
  and consider $c = c_i$.
 By Proposition~\ref{prop:keeping}, there is a Zariski open subset $U$ such that for any $s\in U\cap S(\Q^{\rm alg})$, $f_s$ admits infinitely many preperiodic points on $C_s$. Therefore, 
 by Proposition~\ref{prop:Green_harmonic}, for every such $s$, $c_{s}$ is preperiodic.  Fix $s_0 \in U\cap S(\Q^{\rm alg})$, then 
 $c_{s_0}$ eventually falls on a periodic point $q_{s_0}$. Replacing $f$ by $f^N$ and $C$ by $f^N(C)$ for some $N$, we may 
 assume that $q_{s_0}$ is fixed and $c_{s_0}   = q_{s_0}$. 
 Enlarging $R$ again if necessary we way assume that $q_{s_0}$ is the specialization at $s=s_0$ of a fixed point $q\in \P^2(R)$ of $f$.

Our purpose is to show that $c=q$. To simplify notation we write $\hat{f}= f|_{L_\infty}$. 
Note that the multiplier $\mu:= \hat{f}'_{s_0}(q_{s_0})$ belongs to $\Q^{\alg}$. 
It follows from Kronecker's theorem that either $\mu$ is a root of unity 
or there is a place $v$ on $\Q^{\alg}$ such that $|\mu|_v<1$.

\subsubsection*{Case 1}
$\mu$ is not a root of unity. 

Fix a place $v$ on $\Q^{\alg}$ such that  $|\mu|_v<1$, and consider the completion $\C_v$
of $(\Q^{\alg},|\cdot|_v)$. We then argue in the analytic topology
in the Berkovich analytification of $\P^2_{\C_v}$ and $S_{\C_v}$. 

Fix  a neighborhood $W$ of $s_0$ in $S^{\an}_{\C_v}$ such that for $s\in W$,  $q_s$ is attracting, 
and  a neighborhood   
$V$ of $q_s$ in $L_\infty$ independent of $s\in W$ such that $\hat{f}_s(V)\subset V$ and 
for any $z\in V$, $f^n_s(z)$ converges to $q_s$ as $n\to\infty$. Reducing $W$ if necessary we may assume that for any 
$s\in W$, $c_s$ belongs to $V$. For $s\in W\cap S(\Q^{\rm alg})$, $c_s$ is preperiodic and converges to $q_s$, 
so it is preperiodic to $q_s$, that is, there exists a minimal  $k=k(s)$ such that $f_s^{k(s)}(c_s) = q_s$.  Now we use an argument similar to that of Proposition~\ref{prop:keeping}: let $H\subset S$ be a hypersurface such that the local degree of $f_s$ at $q_s$ is locally minimal outside $H$ and fix $s_1 \in W\setminus H$. Then, there is a neighborhood $W_1$ of $s_1$ in $W\setminus H$
and a neighborhood $V_1\subset V$ of $q_1$ such that for any $s\in W_1$, 
$f_{s}(V_1)\subset V_1$ and $f_s\inv(q_s)\cap V_1 = \set{q_s}$. From this it follows that the only point eventually falling onto 
$q_s$ in $V_1$ is $q_s$ itself. Therefore if  $s\in W_1\cap S(\Q^{\rm alg})$ is so close to $s_1$ that 
$c_s \in V_1$, we infer that  $c_s = q_s$, and finally  $c=q$ by analytic continuation. 

\subsubsection*{Case 2}
$\mu$ is   a root of unity. 

To deal with this case we embed $\Q^{\alg}$ into $\C$, and work at the complex place. Recall that a holomorphic family $(g_\lambda)_{\la\in \Lambda}$ of rational maps on $\P^1(\C)$, parameterized by a connected complex manifold is trivial if any two members are conjugate by a Möbius transformation, depending holomorphically on $\Lambda$. If $c$ is persistently preperiodic we are done, so assume that 
$c$ is not persistently preperiodic.  

Under our assumptions,   
there is a dense set $S(\Q^{\rm alg})$ of parameters such that $c_s$ is preperiodic, but $c$ 
is not persistently preperiodic.  Thus by Chio-Roeder~\cite[Theorem 2.7]{chio-roeder} (see also~\cite[Theorem 4]{preper})
every such parameter belongs to the 
 bifurcation locus of the marked  family $(f_\infty, c)$ (note that $c$ is \emph{not} a critical point here). 
 As a consequence,  the bifurcation locus  of the   family is equal to $S^{\an}_\C$. 
 (Note that it is enough to work in a neighborhood 
 of $s_0$, away from possible singularities of $S^{\an}_\C$.)
 
 A first possibility is that  the family $(\hat{f}_s)_{s\in S^{\an}_\C}$ is non-trivial. 
Then Gauthier~\cite[Theorem A]{gauthier:pairs}
implies that 
 $J(\hat{f}_s)  = L_\infty$ for all $s$. But  since  $\hat{f}_{s_0}$ has a rationally indifferent fixed point,   it admits an attracting petal and $J(\hat{f}_{s_0})\neq L_\infty$. This contradiction shows that  the family 
$(\hat{f}_s)_{s\in S^{\an}_\C}$ is trivial. 

Now the situation is that there is a holomorphic family $\varphi_s$ of Möbius transformations 
such that   $ \varphi_s  \hat{f}_s\varphi_s \inv= g$ is  a fixed rational map $g$ on $\P^1$ with a rationally 
 indifferent fixed point at 0. After this conjugacy, the marked family $(\hat{f}, c)$ becomes 
 $(g, \varphi(c))$. Since $c$ coincides with $q$ at $s_0$ and 
  by assumption $c$ is not persistently preperiodic, there is an open set $\Omega$ 
  of parameters such that for $s\in \Omega$, $\varphi_s(c_s)$ belongs to some  attracting petal associated to 0 for $g$. This contradicts the fact that $\varphi_s(c_s)$ must be preperiodic for a dense set of parameters, and the proof is complete. 
\end{proof}

\begin{proof}[Conclusion of the proof of Theorem~\ref{thm:DMM}]
By Proposition~\ref{prop:infinity_periodic}, any point at infinity of $C$ is preperiodic, 
and, by assumption, one of these points, say $c$, is preperiodic to a non-superattracting periodic point $p$. 
Replace $f$ by $f^N$ and $C$ by $f^N(C)$ for some $N$, so that $c=p$ is fixed. By Proposition~\ref{prop:keeping}, there is a 
non-trivial Zariski open subset $U\subset S$ such that for every $s\in U\cap S(\Q^{\rm alg})$, 
$C_s$ contains infinitely many preperiodic points. Then, since $p_s$ is fixed and not superattracting for $f_s$, 
Theorem~\ref{thm:main_precised} asserts that $C_s$ is preperiodic, and more precisely fixed,
under $f_s$ (see Remark~\ref{rmk:fixed}).  The density  of   $U\cap S(\Q^{\rm alg})$ in $S$ 
(for the Zariski or analytic topology) then implies that $f(C) = C$, and the proof is complete. 
\end{proof}

\bibliographystyle{plain}
\bibliography{biblio-DMM}

\end{document}